\documentclass[12pt,a4paper]{amsart}

\usepackage{amssymb, amstext, amscd, amsmath, amsfonts, amsthm, amscd, color}
 
\usepackage[mathscr]{euscript}  
\usepackage{tikz,mathdots,enumerate}
\usepackage{graphicx}
\usepackage{tikz-cd}

\usepackage{pb-diagram}

\usepackage{pgfplots}
\usepackage{url}

\newtheorem{thm}{Theorem}[section]

\newtheorem{cor}[thm]{Corollary}

\newtheorem{lemma}[thm]{Lemma}

\newtheorem{prop}[thm]{Proposition}

\numberwithin{equation}{section}

\theoremstyle{definition}

\newtheorem{example}[thm]{Example}
\newtheorem{definition}[thm]{Definition}
\newcommand{\bR}{{\mathbb{R}}}

  \newcommand{\E}{{\mathcal{E}}}
  \newcommand{\F}{{\mathcal{F}}}
  \newcommand{\G}{{\mathcal{G}}}

\renewcommand{\L}{{\mathcal{L}}}
  \newcommand{\M}{{\mathcal{M}}}
  \newcommand{\N}{{\mathcal{N}}}

\renewcommand{\S}{{\mathcal{S}}}
  \newcommand{\T}{{\mathcal{T}}}
  \newcommand{\U}{{\mathcal{U}}}

\newcommand{\Int}{\operatorname{Int}}
\newcommand{\Ext}{\operatorname{Ext}}

\usepackage{verbatim}

\usepackage{tikz} 
\usetikzlibrary{calc}

\begin{document}

\setcounter{tocdepth}{1}

\title[Partial triangulations of surfaces with girth constraints]%
{Partial triangulations of surfaces with girth constraints}

\author[Stephen C. Power]{Stephen C. Power}


\address{Dept.\ Math.\ Stats.\\ Lancaster University\\
Lancaster LA1 4YF \\U.K. }

\email{s.power@lancaster.ac.uk}

\thanks{
{MSC2020 {\it  Mathematics Subject Classification.}
05C10, 52C25 \\
{  \today} 
}}

\maketitle

\begin{abstract}
Barnette and Edelson have shown that there are finitely many minimal triangulations of a connected compact 2-manifold $\M$. Similar finiteness results are obtained for cellular partial triangulations that satisfy 
various girth inequality constraints for embedded cycles. A characterisation of various $\M$-embedded
sparse graphs 
is given in terms of the satisfaction of higher genus girth inequalities. With this it is shown that there are finitely many contraction-minimal $\M$-embedded graphs that are $(3,6)$-tight or $(3,3)$-tight. 
\end{abstract}

\section{Introduction}
 The following theorem, obtained in 1989, is due to Barnette and Edelson \cite{bar-ede-1}, \cite{bar-ede-2}.

\begin{thm}\label{t:main}
There are finitely many minimal triangulations of a compact 2-manifold.
\end{thm}
The first paper \cite{bar-ede-1} deals with orientable surfaces while the second paper  \cite{bar-ede-2} obtains the general case, resolves an oversight in \cite{bar-ede-1} and gives a more direct proof avoiding Mayer-Vietoris sequences for homology groups.
The proof, which we give in Section \ref{s:theBEthm}, is an elegant elementary induction argument combined with a somewhat subtle genus bound on the number of disjoint nonhomotopic curves, as given in Theorem \ref{t:loopsgenusbound}. 

We obtain analogous results for contraction-minimal embedded graphs in families of partial triangulations $G\subset \M$ that satisfy girth inequalities. 
In the simplest case these inequalities, the planar type girth inequalities, 
are length constraints on embedded cycles in $G$ that bound an open disc in $\M$ that is not fully triangulated (Definition \ref{d:planarGI}). 
In Section \ref{s:contractionminimalandgirth} it is shown that there are finitely many contraction minimal embedded graphs satisfying the planar type girth inequalities.
In Section \ref{s:generalGIand}, which is independent of Section \ref{s:contractionminimalandgirth}, higher genus girth inequalities are defined in terms of sets of closed walks in $G$ that are the boundary walks of a connected open set of general genus. Theorem \ref{t:Gr_equals_F} 
shows that $G$ is $(3,6)$-tight, in the sense of Definition \ref{d:36tight}, if and only if all higher genus girth inequalities are satisfied, together with the Maxwell count $3v-e=6$ for $G$.
In the final section we show that
for each 2-manifold $\M$ there are finitely many contraction-minimal cellularly embedded 
$(3,\alpha)$-tight graphs for $\alpha=6,3$, 
and we give applications to the rigidity of bar-joint frameworks in $\bR^3$ for Euclidean and nonEuclidean norms.

 A compact 2-manifold $\M$, as a topological space, is homeomorphic to a compact surface without boundary. We assume throughout that $\M$ is  connected. A face of an embedded (finite) simple graph $G$ in $\M$
is a connected component of the complement of $G$ and
we consider a \emph{triangulation} of $\M$ to be an embedded simple graph in $\M$ where every face is homeomorphic to an open disc and is bounded by a 3-cycle of embedded edges. Such an embedded graph $G$ is a \emph{minimal triangulation}, or is \emph{contraction-minimal}, if no edge can be topologically contracted, in the natural way, to give a triangulation of $\M$  with  2 fewer faces. 
Two triangulations are regarded as the same if their embeddings are unique up to a homeomorphism.
The embedded graph definition of a triangulation does not require the stricter condition that each face is uniquely associated with its boundary 3-cycle. Thus the unique minimal triangulation of the sphere in our sense is $K_3$ rather than $K_4$.  
We also note that Barnette \cite{bar} has shown that there are 2 minimal triangulations of the real projective plane.

One can view the $(3,6)$-tightness of a graph as a uniform sparsity condition whereby $3v-e\geq 6$ holds for all embedded subgraphs with $v\geq 3$, and where the Maxwell count holds for $G$  itself. 
On the other hand the satisfaction of the planar type girth inequalities is a nonuniform sparsity condition depending only on those subgraphs associated with the exterior (or interior) of planar type closed walks. (See Lemma \ref{l:Ext_G(c)}.)
An embedded graph in the sphere is well-known to be $(3,6)$-tight if and only if it is a triangulation of the sphere. For surfaces with nonzero genus a cellularly embedded graph  satisfies the Maxwell count if and only if the boundary lengths $l_1, \dots , l_n$ of its nontriangular faces satisfy $\sum_i(l_i-3) = 6g_r(\M)$ where $g_r(\M)$ is the reduced genus of $\M$. (See Lemma \ref{l:alphacyclebounds}.) The Maxwell count therefore ensures that  the number of nontriangular faces depends only on the genus. 

Classes of sparse graphs
admitting cellular embeddings, and their construction by vertex-splitting moves, are topics of current interest in the rigidity theory of generic bar-joint frameworks, and this provides motivation for our considerations here. The rigidity connection stems from the fact that vertex-splitting moves preserve the generic rigidity of such a framework (Whiteley \cite{whi}). See, for example, 
\cite{cru-kit-pow-1}, \cite{cru-kit-pow-2}, 
\cite{cru-kas-kit-sch}, 
\cite{dew-kit-nix}, 
\cite{jor-tan}, and 
\cite{kas-pow}. 

Girth inequalities were introduced in Cruickshank, Kitson and Power \cite{cru-kit-pow-1} in the setting of block and hole graphs considered earlier by Finbow-Singh and Whiteley \cite{fin-whi}. 
These graphs arise from surgery on a triangulated sphere whereby 
the interiors of some essentially disjoint (ie. nonoverlapping) triangulated discs 
are deleted and some of the resulting holes have so-called rigid block graphs attached at their boundaries. In the case of a graph $G$ with a single block and single hole and equal boundary lengths $r\geq 4$, a characterisation of generic rigidity in $\bR^3$ was given in \cite{fin-whi} in terms of the existence of $r$ disjoint paths from the boundary of the hole to the boundary of the block. In \cite{cru-kit-pow-1} it was noted that this is  equivalent to $|c|\geq r$ for any closed walk $c$ separating the block from the hole.
It was also shown that this girth condition is equivalent to the (3,6)-tightness of $G$ if the inserted block graph is minimally rigid, and moreover that this characterisation holds for the case of a single block and several holes. 

Recall that the \emph{girth of a graph} that is not a tree  is the smallest cycle length. The contrasting  terminology above derives from the fact that for the single block and single hole graph it is the cycle lengths $|c|$ that
determine (3,6)-tightness. The condition is that the minimum of these lengths (the ``girth" of the triangulated cylinder) should not be less than $r$.


Other proofs of Theorem \ref{t:main} by Gao et al \cite{gao-et-al}, Nakamoto and Oka  \cite{nak-ota},  and Joret and Wood \cite{jor-woo}, exploit the deep additive formulae for the genus of a graph and obtain upper bounds for the size of a minimal.
We also note that Theorem \ref{t:main} has been generalised in  Boulch et al \cite{bou-et-al} to surfaces with boundary. The sketch proof of Lemma 4 of \cite{bou-et-al} needed for this may be completed with the loop extension theorem \cite{bar-ede-2} 
given in Theorem \ref{t:extensionlemma} below.  

 For  a broad background on embedded graphs see Gross and Tucker \cite{gro-tuc} and Mohar and Thomassen \cite{moh-tho}. Theorem \ref{t:main} has been generalised by Malni\u c and Nedela \cite{mal-ned} to so-called $k$-minimal triangulations $T$ for $k\geq 3$, meaning that the edge-width of $T$ is $k$ and each edge lies on an essential $k$-cycle. See also Theorem 5.4.1 of \cite{moh-tho}. A triangulation has \emph{edge-width k} if the minimal length of an essential cycle is $k$. A cycle, or closed walk of edges, in an embedded graph is \emph{essential} if it is not null-homotpic. 

\section{The Barnette-Edelson theorem.}\label{s:theBEthm}
We give a proof of Theorem \ref{t:main} that follows the elegant proof in \cite{bar-ede-2}. 
The induction scheme of the proof is outlined in the next paragraph and we use a similar inductions for Proposition \ref{p:bound_on3cyclesG_pl} and the main theorem, Theorem \ref{t:mainGgamma}.

\medskip

\noindent \emph{Proof scheme.} A 3-cycle of $T$ is \emph{planar}  if it is  the boundary of an embedded open disc in $\M$ and is \emph{nonplanar} otherwise. 
The surface $\M$ is cut by a nonplanar 3-cycle of the triangulation $T$. If this 3-cycle has a small closed  neighbourhood that is an annulus then the resulting 1 or 2 open manifolds, $\M'$, or $\M'$ and $\M''$, determine compact surfaces with boundary, each component of the boundary having an embedded 3-cycle.
If the cutting 3-cycle has a M\"{o}bius strip neighbourhood in $\M$ then  the resulting compact surfaces with boundary have an embedded 6-cycle in each boundary component. 
In both cases, by capping the boundary components with a 
topological disc one obtains 1 or 2 compact surfaces. These  lower genus surfaces, say $\M_1$, or $\M_1$ and $\M_2$, have triangulations, $T_1$, or $T_1$ and $T_2$, that are derived from $T$ and from small triangulations of the capping discs.
A capping disc for the cutting cycle has the form of an attached triangle or an attached triangulated 6-cycle. For a minimal triangulation $T$, the derived triangulations need not be minimal and the crux of the proof is to show that they are close to minimal in the following sense. The number of edges of $T_1$ in $\M_1$ that do not lie on an essential 3-cycle in $\M_1$ is bounded (with bound  depending only on the genus of $\M$).
These constraints ensure that the derived triangulations are close to minimal, in the sense of Lemma \ref{l:simpleLemma}, where $n$ has a genus bound.
\medskip

The following lemma is Lemma 6 of \cite{bar-ede-1}. See also Lemma 6 of \cite{bou-et-al}. A \emph{shrinkable}, or \emph{contractible}, edge of the triangulation $T$ is one that does not lie on a 3-cycle of edges other than the two facial 3-cycles that contain the edge. In this case there is a natural contraction $G/e$ of $G$, obtained by contracting $e$ to a single vertex and contracting each incident face to an edge. 

\begin{lemma}\label{l:4impliesShrinkable}
{If $e$ is an edge in a triangulation $T$ of a compact 2-manifold $\M$ and if $e$ lies
on four distinct nonplanar 3-cycles, all homotopic to each other, then $T$ has a shrinkable edge.}
\end{lemma}

The idea of the proof given in \cite{bar-ede-1} is that the hypotheses imply there is a triangulated planar patch of $T$ with an edge that is sufficiently interior to this patch that it cannot lie on a nonfacial 3-cycle. 
In Lemma \ref{l:lensLemma} we give a version of this principle which is applicable in our context of edges lying on general critical embedded cycles as well as nonplanar 3-cycles.

It is well-known that a nonfacial planar 3-cycle $c$ with associated open disc $U$, in a triangulation of $\M$,  implies the existence of a contractible edge that is inside $c$. We say that an edge is \emph{inside} $c$ or $U$ if it meets $U$. See Lemma 1 of Barnette \cite{bar} for example. Thus, in a minimal triangulation all nonfacial 3-cycles are nonplanar.

The next theorem gives a key topological fact. See also Malni\u c and Mohar \cite{mal-moh}.

\begin{thm}\label{t:loopsgenusbound}
If $\M$ is not the sphere or projective plane then any family $\E$ of homotopically nontrivial simple pairwise nonhomotopic curves, meeting only at a common point in $\M$, has at most $6g - 3$ members when $\M$ is orientable, and $3g$
members when $\M$ is nonorientable, where $g$ is the genus of $\M$.
\end{thm}

The theorem is Corollary 1 of Barnette and Edleson  \cite{bar-ede-2} and follows from a  topological extension theorem (Theorem 1 of \cite{bar-ede-2}) that corrects the oversight in \cite{bar-ede-1}. This extension theorem is Theorem \ref{t:extensionlemma} below. 

Recall that the genus $g$ of $\M$ is  the maximum number of disjoint loops in $\M$ whose union has a connected complement in $\M$. We say that a set associated with an embedded graph in a family of $\M$-embedded graphs, has a \emph{genus bound} or \emph{independent bound}, if there is an upper bound for the size of the set for this family that depends only on the genus of $\M$.


With the following corollary 
we may bypass some lemmas from \cite{bar-ede-1}. It plays a key role in the induction step of the proof scheme.

\begin{cor}\label{c:degreegenusbound} 
There is a constant $c$ such that the degree of a vertex in a minimal triangulation $T$ of the surface  $\M$ is no greater than $cg$ where $g$ is the genus of $\M$. 
\end{cor}

\begin{proof}
Suppose that $v$ has degree $m$. Consider the graph $N(v)$ corresponding to  the union of the triangles of $T$ incident to $v$. By minimality there are at least $\lceil m/2 \rceil$ nonplanar 3-cycles through $v$ with various edges not in $N(v)$ but with vertices adjacent to $v$. We call these  \emph{peripheral edges}. 
By Lemma \ref{l:4impliesShrinkable} there must be at least $ \lceil m/2 \rceil/3$ of these 3-cycles that are pairwise nonhomotopic. 
By shrinking $N(v)$ to $v$, by a path of homeomorphisms of $\M$, we see that the peripheral embedded edges of these 3-cycles determine a set of loops at $v$ which are otherwise disjoint. These loops are also pairwise nonhomotopic and so, by Theorem \ref{t:loopsgenusbound}, $m\leq cg$ for some constant $c$.
\end{proof}

We next give a proof of Theorem \ref{t:main} assuming Theorem \ref{t:loopsgenusbound} and follow this with a proof of Theorem \ref{t:loopsgenusbound}. 
\medskip

Let $T'$ be a triangulation obtained from the triangulation $T$ by shrinking a single edge. Then the reverse operation, $T \to T'$, is referred to as a \emph{vertex-splitting move}, or, more precisely, a \emph{planar vertex-splitting move}.  The following simple lemma plays a role in the induction argument of the proof.
It also features in Theorem \ref{t:partialtriangulations}.

\begin{lemma}\label{l:simpleLemma}
Let $T$ be a triangulation of a compact manifold $\M$ for which there are exactly $n$ edges that do not lie on a nonfacial 3-cycle. Then $T$ is obtained from a minimal triangulation by a sequence of at most $n$ vertex-splitting moves. 
\end{lemma}

\begin{proof}Observe first that if $T\to T'$ is an edge shrinking move between triangulations and some edge $f$ in $T$ lies on a nonfacial 3-cycle of $T$, then the corresponding edge in $T'$ once again lies on nonfacial 3-cycle. Thus shrinking an edge that does not lie on a nonfacial 3-cycle strictly decreases the number of edges failing to lie on a nonfacial 3-cycle. The edge shrinking move and the planar vertex-splitting moves are inverse operations on the set of triangulations of $\M$ and so the lemma follows.
\end{proof}

\begin{proof}[Proof of Theorem \ref{t:main}]
Let $\M$ be orientable with genus $g$ and suppose that the number of minimal triangulations of a compact 2-manifold with lower genus  
is finite.
Cutting $\M$ with a 3-cycle that is not homotopic in $\M$ to a point 
we obtain 1 or 2 manifolds with boundary. Capping boundaries, as indicated above, yields a compact 2-manifold $\M_1$, or the pair $\M_1$ and $\M_2$, each of which has genus smaller than $g$. 

Suppose first that there is one derived manifold $\M_1$. It carries an inherited triangulation $T_1$ which need not be minimal. A shrinkable edge, $f$ say, of $T_1$ has the property that it lies on a 3-cycle in $\M$ which passes through a vertex of the cutting 3-cycle.
By Lemma \ref{c:degreegenusbound} the degrees of these vertices, $x,y$ and $z$, in the triangulation $T$, are bounded by $cg$ for some constant $c$. 
It follows that the number of such 3-cycles, and therefore the number of shrinkable edges $f$, is at most $dg$ for some positive constant $d$. Assume, for the induction hypothesis, that the number of minimal triangulations of $\M_1$ is finite. By
 Lemma \ref{l:simpleLemma}  the number of triangulations $T_1$  of $\M_1$ obtained by the construction above is finite and so
the number of minimal triangulations $T$ is finite. 
This argument establishes the induction step in the orientable case with cut cycle giving a connected manifold $\M_1$. The same argument applies when the cut leads to a pair $\M_1, \M_2$ and so, since there is one minimal triangulation for the sphere the theorem follows for orientable 2-manifolds. 

In the nonorientable case the capping of the boundaries of $\M_1$ (or $\M_1$ and $\M_2$) is  by  the attachment of triangulated discs to 3-cycles or 6-cycles and the argument is the same. 
\end{proof}

The proof of Theorem \ref{t:loopsgenusbound} depends on the following extension theorem which shows that we may assume that
the (nonsimple) embedded graph associated with $\E$ is \emph{cellular}.


We make use of the following terms. An embedded graph $G$ in a compact 2-manifold $\M$ is \emph{cellular} if each of its faces is homeomorphic to an open disc. The \emph{genus} $g(U)$ of a general face $U$ is the maximum number of disjoint loops in $U$ that do not disconnect $U$. 
A loop edge $e$ in a surface is  \emph{essential} if an associated simple closed curve for it, say $\alpha(t), 0\leq t\leq 1$, is not null-homotopic. Also, a pair of loops $e_1, e_2$ with curves $\alpha, \beta$ with  $\alpha(0)=\beta(0)=v$ is a homotopic pair of loops if  $\{\alpha, \beta\}$ or $\{\alpha, -\beta\}$ is a homotopic pair.
Loops in a surface are locally 2-sided. If both sides of a loop edge $e$ belong to the same face we say that $e$ is a 1-sided edge, otherwise $e$ is called a 2-sided edge. By the number of edges of a face $U$ we  mean the number of 2-sided edges in
the boundary of $U$ plus twice the number of 1-sided edges in the boundary.

\begin{thm}
\label{t:extensionlemma}
Suppose that $\M$ is not the sphere or the projective plane. Then any family $\E=\{e_1, \dots ,e_r\}$ of pairwise nonhomotopic essential loops, meeting only at a common point $v$, extends to a similar family $\tilde{\E}$ whose embedded graph $G(\tilde{\E})$ 
is cellular with each face having at least 3 edges.
\end{thm}

\begin{proof}[Proof sketch]
Let $U$ be a face of the embedded multigraph $G(\E)$ of $\E$. 
The idea of the proof is to introduce a loop in $U\cup \{v\}$ in order to extend $\E$ to a similar such family where $U$ has been  replaced by 1 or 2 faces that have lower genus. 

If $g(U)$ is positive then  
viewing $U$ as an open surface of finite genus, there exists a loop in $U$ that passes through either 
a handle or a cross cap of $U$. 
This loop may modified to give a loop, $e_*$ say, in  $\{v\}\cup U$, with the relative topology, that passes through $v$ and the handle or cross cap. 
It follows that $e_*$ is an essential loop in $\M$. Also it is not homotopic to any loop of $\E$ that is in the closure of $U$. This follows formally on considering the fundamental group of the  closure of $U$ in $\M$. Further considerations shows that $e_*$ is not homotopic to any loop in $\E$.
For the extended set $\E'=\E\cup\{e_*\}$ with embedded graph $G(\E')$ the face $U$  has been replaced by 1 or 2 faces and it follows from the definition of $g(U)$ that they have strictly lower genus.

Repeating this argument sufficiently often leads to an extended set  $\tilde{\E}$ for which the associated embedded graph has all faces 
of genus 0. Further examination, together with the assumption that $\M$ is not the projective plane, shows that each face has at least 3 edges.
\end{proof}

 The general Euler formulae for an embedded graph $G$  are
\[ 
v-e+f=2-2g\quad \mbox{and} \quad  v-e+f=2-g,
\]
for, respectively, the case of orientable and nonorientable $\M$, where  $v,e$ and $f$ denote the number of vertices, edges and faces of $G$. The genus $g=g(\M)$ is defined to be the maximum number of disjoint loops whose union has a connected complement in $\M$. Moreover, it can be defined in terms of the standard models for $\M$ as the number of attached handles to a sphere when $\M$ is orientable, and as the number of attached cross caps to a sphere when $\M$ is not orientable.

\medskip

\begin{proof}[Proof of Theorem \ref{t:loopsgenusbound}.]
Let $|\E|=n$. By the previous theorem we may assume that the embedded graph determined by $\E$, with a single vertex and $n$ embedded edges, is cellular and the boundary of each face is the union of at least 3 embedded edges. 
By Euler's formula, for orientable $\M$, we have $1-e+f = 2-2g$. Also, $2e= \sum_k kf_k$ where $f_k$ is the number of faces with $k$ edges. Thus $2e \geq 3f,$ 
\[
e = 1+f+2g-2 \leq 2e/3 +2g-1
\]
and so $e\leq 6g-3$.  The genus bound in the nonorientable case is obtained in the same way.
\end{proof}

\section{Sparse graphs and girth inequalities}\label{s:sparseGraphs}

Our main interest concerns various families of embedded sparse graphs that satisfy the Maxwell count $3v-e=6$. However we also consider here families with $3v-e=\alpha$ for general  values of $\alpha$.

\begin{lemma}\label{l:alphacyclebounds} 
Let $G$ be a cellular embedded graph in  a compact 2-manifold $\M$ that  satisfies the global count $3v- e = \alpha$. 
Then
\[
\sum(k-3)f_k= \alpha+3\mu g-6
\] 
where $\mu=2$ if $\M$ is orientable and $\mu=1$ otherwise.
\end{lemma}

\begin{proof} We have $2e= \sum_k kf_k$ and $3(2-\mu g)=3v-3e+3f$ and so
\[
6-3\mu g = e+\alpha -3 e + 3\sum f_k = \alpha   -\sum kf_k+ 3\sum f_k,
\]
and the identity follows.
\end{proof}

In particular, for such cellular graphs there is a bound on both the size and number of faces with more than 3 edges and this bound  depends only on $\alpha$ and the genus of $\M$.

Let $\F(\M)$ be a family of simple embedded graphs for the compact 2-manifold $\M$. 
Then $G$ in $\F(\M)$ is said to be \emph{contraction-minimal} if for every edge $e$ belonging to two facial 3-cycles, and no other 3-cycles, the contracted embedded graph $G/e$ does not belong to $\F(\M)$. 

The Barnette-Edelson theorem shows that there are finitely many contraction-minimal embedded graphs in the family $\T(\M)$ of triangulations of $\M$. For a simple corollary, consider the family $\T(\M,n_4,\dots ,n_r)$ of \emph{partial triangulations} that are cellular embedded simple graphs $G$ in $\M$ that have $n_k$ faces with closed boundary walks of length $k$, for $4\leq k\leq r$.
We extend earlier terminology by defining a 3-cycle of $G$ to be  \emph{planar} if it is the boundary cycle of an open disc in $\M$ that contains no nontriangular faces of $G$. 
In particular, a 3-cycle that is not planar is an essential 3-cycle.


\begin{thm}\label{t:partialtriangulations} There are finitely many  contraction-minimal embedded graphs in the family 
$\T(\M,n_4,\dots , n_r)$.
\end{thm} 

\begin{proof}
Let $G$ be contraction-minimal in $\T(\M,n_4,\dots ,n_r)$. Then every edge on 2 facial 3-cycles of $G$ lies on a nonfacial 3-cycle of $G$. Add a vertex and $k$ edges to each face of $G$ with a closed boundary walk of length $k\geq 4$ to obtain an embedded simple graph $G^+$ in $\T(\M)$. Let $\S$ be the set of these added edges together with the set of edges that belongs to these boundary walks of $G$.
Then $|\S|$ is no greater than $N=2n_4+\dots +2n_r$. Any other edge of $G^+$ is an edge of $G$ in the boundaries of 2 adjacent facial 3-cycles  and so must lie on a nonfacial 3-cycle of $G$. 
It follows from Lemma \ref{l:simpleLemma} that $G^+$ is obtained from a minimal triangulation of $\M$ by at most $N$ vertex-splitting moves and so, by the Barnette-Edelson Theorem, the embedded graphs $G^+$, and hence the embedded graphs $G$, are finite in number.
\end{proof}



\subsection{Girth inequalities.}
Let $c$ be a closed walk in an $\M$-embedded simple graph $G$. Then
$c$ is a \emph{planar type closed walk} if, as an embedded graph, $c$ has a face $U$, homeomorphic to the open disc, for which $c$ is the boundary walk. We refer to $U$ as a \emph{cellular face}. 

\begin{definition}\label{d:cAndGirthIneq} Let $G$ be a simple embedded cellular graph for the compact 2-manifold $\M$ with nontriangular faces $W_1, \dots , W_r$.
Then a planar type closed walk  $c$ in $G$ with cellular face $U$ containing a subset of the nontriangular faces \emph{satisfies the girth inequality} if
\[
|c|-3 \geq \sum_{k\in I(U)}(|c_k|-3),
\]
where $I(U)$ is the index set for the boundary walks $c_k$ of the  faces $W_k$ that are contained in $U$. When equality holds $c$ is said to be a \emph{critical planar type closed walk for $G$}. 
\end{definition}

\begin{definition}\label{d:planarGI}
A cellular embedded graph $G$ in a compact 2-manifold $\M$ satisfies the \emph{planar girth inequalities} if every planar type closed walk $c$ in $G$ with cellular face $U$ satisfies the girth inequality. 
\end{definition}


Write $\G_{\rm pl}(\M,\alpha)$ for  the family of simple cellular embedded graphs $G$ in $\M$ that satisfy the planar girth inequalities and the freedom count $f(G)=3v-e=\alpha$.
Lemma \ref{l:alphacyclebounds} implies that for values $\alpha<6$ this family can be empty for low genus surfaces $\M$.


For a planar type walk $c$ with open disc face $U$ let $\Int_G(c)$, the \emph{interior graph} of $c$, be the subgraph of $G$ formed by the union of the edges of $c$ and the edges of $G$ that meet $U$. Also let $\Ext_G(c)$, the \emph{exterior graph} of $c$, be the subgraph of $G$ whose edges and vertices do not meet $U$. 

\begin{lemma}\label{l:Ext_G(c)}
Let $c$ be a planar type closed walk for the cellular embedded graph $G$ with $f(G)=\alpha$. Then the following assertions are equivalent.
\medskip

(i) $c$ satisfies the girth inequality.

(ii) $f(\Ext_G(c))\geq \alpha.$
\medskip

In particular $c$ is a critical planar type closed walk for $G$ in $\G_{\rm pl}(\M,6)$ if and only if $f(\Ext_G(c))=6.$
\end{lemma}

\begin{proof}
By Lemma \ref{l:alphacyclebounds} we have
\[
\sum_{k\geq 4}(k-3)f_k= 3\mu g +f(G)-6
\] 
and so
\[
\sum_{k\in I(U)}(|c_k|-3) +\sum_{k\notin I(U)}(|c_k|-3)= 3\mu g+f(G)-6.
\] 
For the cellular embedded graph $\Ext_G(c)$ we have, similarly,
\[
|c|-3 +\sum_{k\notin I(U)}(|c_k|-3)= 3\mu g +f(\Ext_G(c))-6
\] 
and so the equivalence follows.
\end{proof}


\begin{definition}\label{d:36tight} 
A graph $H$ is \emph{$(3,6)$-sparse} if it satisfies the local count 
$f(H')\geq 6$ for each subgraph $H'=(V',E')$ with at least 3 vertices and is 
\emph{$(3,6)$-tight} if in addition $f(H)=6$. The family $\F(\M,6)$ is the family of cellular embedded graphs $G$ in $\M$ whose underlying graphs are $(3,6)$-tight.  
\end{definition}

It follows from the lemma that if $G$ is $(3,6)$-tight (meaning that the underlying graph has this property) then $G$ belongs to $\G_{\rm pl}(\M,6)$. 
The equality $\F(\M,6)=\G_{\rm pl}(\M,6)$ does not hold in general as the following example shows.

By the formula above if the torus is partially triangulated with a single nontriangular face with a closed walk $c$ of length $r$ then $f(G)= r-3$. Thus $r=9$ if $G$ is in $\F(\M,6)$. Figure \ref{f:doubledonut1} shows a double torus $\M$, with $g(\M)=2$, and two excised open discs (shaded).
We consider triangulations $G$ of this surface with boundary wherein the two nontriangular faces have  
boundary cycles of lengths 8 and 10 and where a 3-cycle subgraph $K_3$ of $G$ is embedded as an essential 3-cycle, as shown. The embedded graph $G$ can be viewed as the join of two partial triangulations $G_1, G_2$ of two tori, where the underlying graph $H$ of $G$ is the join over a common 3-cycle of the underlying graphs $H_1, H_2$ of $G_1, G_2$. We have $f(G)=6$ and $f(G_1)=5, f(G_2)=7$, and so any such embedded graph $G$ is not $(3,6)$-tight.

We may, in addition, specify the triangulation $G$ by triangulatiions of $G_1$ and $G_2$ so that 
there is no planar type closed walk in $G_1$ (resp. $G_2$) around the single nontriangular face with length smaller than 8 (resp. 10). In this case note that if a closed walk in $G$ containing just the 8-sided face has a subwalk in $G_2$, then we may assume the subwalk is maximal and starts and ends at vertices in the embedded $K_3$. This implies that there is a shorter such path entirely in $G_1$. Similar remarks apply to closed walks around the 10-sided face. This means that the 
girth inequalities for $G$ can only fail if there is a planar type closed walk whose face contains both nontriangular faces and whose length is less than 15. However, such a walk includes 2 of the vertices of the common 3-cycle and so, by our assumption for $G_1, G_2$, its length is at least (8-1)+(10-1)=16. Thus the girth inequalities hold and we conclude that $\G_{\rm pl}(\M,6)$ is strictly larger than
 $\F(\M,6)$.
 \begin{center}
\begin{figure}[ht]
\centering
\includegraphics[width=7cm]{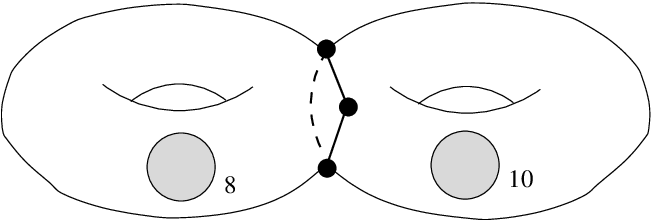}
\caption{Nontriangular faces of a cellular graph in the double torus with boundary cycles of lengths 8 and 10.} 
\label{f:doubledonut1}
\end{figure}
\end{center}


\subsection{Critical patches and contractible edges} 
In Lemma \ref{l:lensLemma} we obtain a bound for the size of certain triangulated discs in a contraction-minimal embedded graph $G$ satisfying girth inequalities. It plays a role for critical closed walks similar to that of Lemma \ref{l:4impliesShrinkable} for essential 3-cycles. 
The following preliminary lemma is due to Fisk and Mohar \cite{fis-moh}. See also Lemma 5.4.2 of \cite{moh-tho}. For completeness we give the short proof. 

\begin{lemma}\label{l:fiskMohar}
Let $k\geq2, r\geq 2$ and let $H_k$ be a simple graph that is the union of the edges and vertices of a set of paths $\pi$ between vertices $a$ and $b$ where the initial edges of the paths are distinct. If the size of the set, $N(k)$, is at least $r^{k-1}(k-1)!$ then there is a vertex $z$ and $r$ paths $\pi$ with internally disjoint subpaths from $a$ to $z$.
\end{lemma}

\begin{proof}
Note that $N(2)=r$. Suppose, by way of induction, that $N(k)$ exists as required. Let $\pi_*$ be a path of $H_{k+1}$ and suppose that there are $kN(k)$ paths for $H_{k+1}$ that meet an  internal vertex of $\pi_*$ (distinct from $a, b$). Then there are $N(k)$ paths for $H_{k+1}$ that meet a common internal vertex $z$  and so the existence of the desired $r$ subpaths for $H_{k+1}$ holds in this case. 
Now let $\pi_1, \dots , \pi_q$ be a maximal subset of the paths for $H_{k+1}$ that are internally disjoint. If $q\geq r$ this subset suffices for the conclusion so we may assume $q< r$. Since every path in the set for $H_{k+1}$ has, by the maximality, a common internal vertex with one of the paths $\pi_j$, by the previous case if $N(k+1) = rkN(k)$ the desired conclusion holds. This completes the proof by induction.
\end{proof}

We assume now that $G$ is a contraction-minimal graph in $\G_{\rm pl}(\M,\alpha)$.
By Lemma \ref{l:Ext_G(c)} every edge $e$ that lies on two facial 3-cycles, either lies on an essential  3-cycle or lies on a critical planar type closed walk. Indeed, since 
there are no planar type 3-cycles in $G$ other than facial 3-cycles (as noted in the introduction) if this is not the case then the contracted embedded graph $G/e$ is simple and the girth inequalities hold.

Let $c_1, c_2$ be essential 3-cycles or a critical planar type walks in $G$, with $G$ in $\G_{\rm pl}(\M,\alpha)$ and 
suppose that $\pi_1$ and $\pi_2$ are subwalks of $c_1, c_2$, respectively, from $v$ to $w$, and that $v, w$ are the only vertices common to $\pi_1, \pi_2$.  Moreover suppose that the concatenation of $\pi_1$ and the reversal of $\pi_2$ gives the  boundary walk of an open disc that contains only triangular faces of $G$. We denote the triangulated subgraph determined by these faces as $P(\pi_1, \pi_2)$ and refer to it as a \emph{critical patch} of $G$.
It follows that $|\pi_1|=|\pi_2|$, since otherwise the substitution of the shorter subwalk for the longer one would give a violation of the girth inequalities. Let the common length be $d$  and note that, for the same reason, every
path $\pi$ of edges in the patch from $v$ and $w$ has length at least $d$. 

\begin{lemma}\label{l:lensLemma} Let $P=P(\pi_1,\pi_2)$ be a critical patch for the contraction-minimal graph $G$. Then there is an independent bound 
for the number of edges in $P$.
\end{lemma}

\begin{proof} 
Suppose that $\rho_1, \dots ,\rho_s$ are paths in $P$ from $v$ to $w$ of length $d$ that are internally disjoint. We first show that there is an independent bound for $s$ (depending only on $g(\M)$ and $\alpha$). 
We may assume that these paths are internally disjoint from $\pi_1, \pi_2$ and are naturally ordered as adjacent paths. Also we may assume that $s\geq 5$. Consider the edge $e=xy$ of a triangular face of the form $vxy$ where $vx$ is an edge of $\rho_i$. Figure \ref{f:lensFigure} indicates such an edge when $s=5, d=5$ and $i=3$.
\begin{center}
\begin{figure}[ht]
\centering
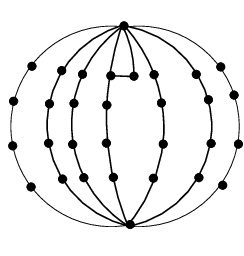
\caption{Internally disjoint paths of length 5 in the critical patch $P(\pi_1,\pi_2)$.} 
\label{f:lensFigure}
\end{figure}
\end{center}
By the contraction-minimal hypothesis $e$ lies on an essential 3-cycle of $G$ or on a critical planar type walk $c$ of $G$. The first possibility is evidently not possible if $1<i<s$. 
Consider then the subwalk $\pi$ of $c$ in the patch $P$ that contains $e$ and has vertices $a, b$ on the boundary of $P$. 
Suppose, for a contradiction, that $\{a, b\}=\{v, w\}$. Since $c$ is critical and $P$ is a critical patch the length of $\pi$ is equal to $d$. Indeed the length cannot be less than $d$ since $P$ is a critical patch and cannot be greater than $d$ since the substitution of any path $\rho_i$ for $\pi$ in $c$ gives a closed walk $c'$ of planar type whose face $U(c')$ contains the same nontriangular faces of $G$ as $U(c)$. Also $|c'|<|c|$ contradicting the girth inequality for $c'$. Thus $\{a, b\}\neq\{v, w\}$. In this case we note first that Lemma \ref{l:alphacyclebounds} implies that there is an independent bound $M$ for the length of $c$. Thus if $s\geq 2M$ and $i=M$ then there is no path of length less than $|c|$ between $a$ and $b$. And so $s$ is less than $2M$, and there is an independent bound for $s$.

It now follows from Lemma \ref{l:fiskMohar} that there is an independent bound for the degree of $v$ in $P$. 
Since the lemma is true for $d=2$ (since all paths of length 2 from $v$ to $w$ are internally disjoint) we may complete the proof by induction. Suppose that the lemma is true for $d$ less than $d+1$ and consider the neighbouring vertices $v_1, \dots , v_m$ of $v$ in $P$ where $m$ is the degree of $v$ in $P$. Then there  are independent bounds for the subpatches $P(v_1,w), \dots , P(v_m,w)$. The size of the complement of the union of these subpatches is at most $2m-1$ and so the lemma is true for $P$.
\end{proof}

Critical planar type walks $c_1, c_2$ for an embedded graph in $\G_{\rm pl}(\M, \alpha)$ are said to be \emph{equivalent} if their open disc faces $U_1, U_2$ contain the same nontriangular faces. With this notation we have the following corollary.

\begin{lemma}\label{l:boundforPinchedBoundary} Let
$G$ be contraction-minimal and suppose that $c_1, c_2 $ are equivalent critical walks that are not disjoint. Then there is an independent bound for the number of edges in $U_2\backslash U_1^-$ and $U_1\backslash U_2^-$.
\end{lemma}

\begin{proof} 
If $c_1$ and $c_2$ have at least 1 common vertex then there is an independent bound for the number of connected components of $U_2\backslash U_1^-$ and $U_1\backslash U_2^-$. Since $c_1, c_2$ are equivalent these components are open discs. Since the closure of each component corresponds to a critical patch the required bound follows from Lemma \ref{l:lensLemma}.
\end{proof}

\begin{lemma}\label{l:maxCritsPatch} Let
$G$ be a contraction-minimal embedded graph  in $\G_{\rm pl}(\M, \alpha)$ and let $H\subset G$ be a triangulated polygon with a vertex $v$ whose distance to the boundary of $H$ is at least 4. Then there is an interior edge of $H$ that does not lie on a critical walk of maximal length.
\end{lemma}

\begin{proof}
Let $va$ be an edge incident to $v$ and let $c$ be a critical walk for this edge containing the edges $av, vb$. Note that $a$ and $b$ cannot be adjacent. We assume first that $v_1, v_2$ are neighbours of $v$ distinct from $a,v,b$ as shown in Figure \ref{f:concentricCriticals}.
\begin{center}
\begin{figure}[ht]
\centering
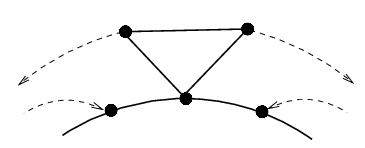
\caption{Subwalks for critical walks for $vv_1, vv_2$.} 
\label{f:concentricCriticals}
\end{figure}
\end{center} By way of contradiction suppose that $vv_1$ and $vv_2$ lie on a critical walks $c_{vv_1}$, $c_{vv_2}$ that have the maximal length $d$. By subwalk substitutions, between $c_{vv_1}$ and $c$ we may suppose  $c_{vv_1}$ contains $bv$ and that it has a subwalk $\pi_b$ of length $d-2$ from $v_1$ to $b$. Similarly we may assume there is also a subwalk $\pi_a$ of  $c_{vv_2}$ from $v_2$ to $a$ of length $d-2$. The critical walks, and therefore the subwalks, cross at a common vertex $u$. Assume that the distances from $v_1$ to $u$ and $a$ to $u$ are $\gamma$ and $\delta$ respectively. Then by the planar girth inequalities, for the blended path from $v_1$ to $u$ to $v_2$ and the blended path from $a$ to $u$ to $b$, we have
\[
\gamma + (d-2-\delta)+1 \geq d, \quad \delta +(d-2-\gamma)+2\geq d
\]
which is a contradiction. Since the configuration of Figure \ref{f:concentricCriticals} is not possible it follows that the degree of $v$ is 4. Since this conclusion also applies to all the neighbouring vertices of $v$ we obtain the desired contradiction since $H$ is a triangulated polygon.
\end{proof}

The next lemma shows how critical walks may be modified by subwalk substitutions. 

\begin{lemma}\label{l:UcapUdash}
Let $G$ belong to $\G_{\rm pl}(\M,\alpha)$ with critical walks $c, c'$ with open disc faces $U, U'$ respectively, and suppose that $U\cap U'$ has only one component, $U_0$, that contains nontriangular faces. Then the boundary walk of $U_0$ is a critical walk.
\end{lemma}

\begin{proof} By subwalk substitutions we may assume that $c, c'$ enclose no critical patches, so that $U_0$ is the intersection of $U$ and $U'$ and the boundary walk $b$ of   $U$ is a concatenation of subwalks of $c$ and $-c'$.
We have  $\Ext_G(b)=\Ext_G(c)\cup \Ext_G(c')$ and, by Lemma   
 \ref{l:Ext_G(c)}, $f(\Ext_G(c))=f(\Ext_G(c'))=\alpha$. Thus  
\[
f(\Ext_G(b))= \alpha+\alpha- f(\Ext_G(c)\cap \Ext_G(c')).
\] 
Since $G$ satisfies the girth inequalities we have $f(\Ext_G(b))\geq \alpha$. Also, $\Ext_G(c)\cap \Ext_G(c')= \Ext(c'')$ where $c''$ is a boundary walk formed from subwalks of $c, c'$.
Thus we also have $f(\Ext_G(c)\cap Ext_G(c'))\geq \alpha$, and so the equation for $f(\Ext_G(b))$ implies that $f(\Ext_G(b))=\alpha$, as desired.
\end{proof}

\begin{lemma}\label{l:maxCritsPatch_inU}
Let $G$ be a contraction-minimal embedded graph  in $\G_{\rm pl}(\M, \alpha)$ and let 
$c$ be a critical walk of length $d$ with open disc face $U$.
Let $H\subset U^-$ be a triangulated polygon with a vertex $v$ whose distance to the boundary of $H$ is at least 4. Then there is an interior edge of $H$ that does not lie on a critical walk in $G$ that is equivalent to $c$.
\end{lemma}

\begin{proof}
In view of Lemma \ref{l:maxCritsPatch} it suffices to show that
there is an interior edge of $H$ that does not lie on a critical walk contained in $U^-$ of length $d$. The proof of this follows  the proof of Lemma \ref{l:maxCritsPatch}. 
\end{proof}

\section{Girth conditions and contraction-minimal graphs}\label{s:contractionminimalandgirth}

In this section we obtain the following finiteness theorem.  
Lemmas \ref{l:degreeboundGirth}, \ref{l:boundforU_silverBullet}, 
\ref{l:indptbound_critwalks} 
also features in the genus zero superface arguments required for the initial step of the induction proof of Theorem \ref{t:mainGgamma}. We assume that $\alpha$ is a positive integer. 

\begin{thm}\label{t:finiteforgirth} Let $\M$ be a compact connected 2-manifold. Then there are finitely many contraction-minimal embedded graphs in $\G_{\rm pl}(\M,\alpha)$.
\end{thm}

For $G$  in  $\G_{\rm pl}(\M,\alpha)$ let  $\M(G)$ be the connected compact topological space given by deleting the nontriangular faces of $G$ from $\M$. In particular, $\M(G)$ need not be a surface with boundary in the classical sense that the boundary consists of finitely many disjoint simple loops. 

\begin{lemma}\label{l:htpy_map}
Let $\L$ be a set of disjoint loops in $\M(G)$, based at $v$,  with homotopy classes $[\gamma]_{\M(G)}$. Then the natural map 
$\theta_\L:[\gamma]_{\M(G)}\to [\gamma]_\M$ from the $\M(G)$-homotopy classes of $\L$ to the $\M$-homotopy classes of $\L$, induced by the inclusion map,  is finite-to-one.
\end{lemma}

\begin{proof} Let $r$ be the number of open discs in $\M$ whose deletion gives $\M(G)$. Suppose $r=1$ with open disc $U$ and that $\gamma_1, \gamma_2$ are loops of $\L$ that are homotopic in $\M$. Then they form the boundary of a pinched annulus in $\M$. If the loops are not homotopic in $\M(G)$ then $U$ is contained in this annulus. If $\gamma_3$ is homotopic to $\gamma_1$ and $\gamma_2$ in $\M$ it follows that it is homotopic in $\M(G)$ to either $\gamma_1$ or $\gamma_2$.
The general case follows similarly by induction. 
\end{proof}

The next lemma is a counterpart to the degree boundedness given in Corollary \ref{c:degreegenusbound} and its proof also makes use of the genus bound of nonhomotic loops obtained by Barnette and Edelson in Theorem \ref{t:loopsgenusbound}. 
In the proof and in this section we refer to a critical planar type boundary walk simply as a critical walk.

\begin{lemma}\label{l:degreeboundGirth} Let $\M$ be a compact 2-manifold and let $G$ be a contraction-minimal embedded graph in the family $\G_{\rm pl}(\M,\alpha)$.
Then there is an independent bound
for the degree of a vertex in $G$. 
\end{lemma}

\begin{proof}
Let $v$ be a vertex. Then there is an independent bound for the number of incident edges to $v$ that lie on the boundary of $G$. This is so since the number and lengths of the boundary walks of the nontriangular faces have independent bounds by Lemma \ref{l:alphacyclebounds}.
 The remaining edges $e$ incident to $v$ lie on an essential 3-cycle $c_e$ or on a critical boundary walk $c_e$. It suffices then to show that there is an independent bound for the number of these edges $e$ for which the cycles or walks $c_e$ are of a particular length $d\geq 3$. 

Let $e_i, 1\leq i\leq s,$ be a choice of these incident edges that have a critical cycle or walk $c_{e_i}$ of length $d$. Discarding at most $\lceil s/2\rceil$ walks or cycles we can arrange that each edge $f_j$ of the remaining edges, relabelled as $f_1, \dots , f_t$, belongs only to $c_{f_j}$. Consider each $c_{f_j}$ as a directed closed path with initial edge
$f_j$. By Lemma \ref{l:fiskMohar}, with $a=b=v$, given a positive integer $r$ if $t\geq r^{(d-1)}(d-1)!$ then there is subset of $r$ closed paths $c_{f_j}$, say $\pi_1, \dots , \pi_r,$, with subwalks between $a$ and a common vertex $z$ that are internally disjoint subwalks. 

Write $[\gamma]_{\M(G),a,z}$ for the homotopy class of the curves $\gamma$ in $\M(G)$ with initial and final points at $a$ and  $z$ respectively. If $\pi_i$ and $\pi_j$ are homotopic in $\M(G)$ then they form the boundary curves of a critical patch, namely $P(\pi_i,\pi_j)$. By Lemma \ref{l:lensLemma}  there is a common  independent bound for the size of the  equivalence classes 
\[
P_j=\{\pi_i: \pi_i\in [\pi_j]_{\M(G),a,z}\},\quad  j=1,\dots , r.
\]

To prove the lemma it suffices to show that there is an independent bound for the size of a subset of the paths $\pi_i$ with distinct homotopy classes $[\pi_i]_{\M(G),a,z}$. For in this case there must be an independent bound for $r$ and hence $t$ and $s$. 

By  Lemma \ref{l:htpy_map},  we need only show that there is an independent bound for the number of distinct homotopy classes $[\pi_i]_{\M,a,z}$.
To this end suppose that $\gamma_1, \dots ,\gamma_q$ are such curves in $\M$ with endpoints $a,z$ that are pairwise nonhomotopic.  Contracting $z$ to $a$ along the curve $\gamma_1$
gives $q-1$ closed curves in $\M$ that are disjoint except for a common vertex. Moreover they are pairwise nonhomotopic and so, by Theorem \ref{t:loopsgenusbound}, have an independent bound, completing the proof.
\end{proof}

\begin{lemma}\label{l:boundforU_silverBullet} Let
$G=(V,E)$ be a contraction-minimal embedded graph  in $\G_{\rm pl}(\M, \alpha)$ and let $c$ be critical planar type walk with face $U$.  Then there is an independent bound for the set $V\cap U$.
\end{lemma}

\begin{proof}
Let $U$ contain the nontriangular faces $W_1, \dots ,W_n$.
If $e$ is an edge of $G$ that lies on a pair of facial 3-cycles then by the hypotheses $e$ lies on a critical walk $c_e$ or an essential 3-cycle $c_e$. An essential 3-cycle $c_e$ necessarily contains a vertex of $c$ and so by the degee boundedness of Lemma \ref{l:degreeboundGirth} the number of such edges $e$ has an independent bound. Similarly the set of edges $e$ for which $c_e$ is a critical walk meeting $c$ has an independent bound. 

To complete the proof we consider the set $S$ of edges $e$ for which $c_e$ is a critical walk that does not meet $c$ and show that it has an independent bound.

Suppose first that $U$ contains a single nontriangular face
$W_i$ with a boundary walk of length $d$, and let $v$ be a vertex that has distance at least 4 from the complement of $S$ in $U$. The proof of Lemma \ref{l:maxCritsPatch} applies and so $S$ has an independent bound.

By way of induction assume that the lemma is true if there are fewer than $n$ nontriangular faces and that $U$ contains $n$ nontriangular faces. Once again let $S$ be the set of vertices in $U$ that do not lie on a critical walk meeting the boundary of $U$ or on an essential 3-cycle. By the induction hypothesis it remains to show that $S_n$ has an independent bound where $S_n$ is the set of vertices of edges $e$ for which $c_e$ is necessarily a critical walk containing all $n$ nontriangular faces. Once again, the proof of Lemma \ref{l:maxCritsPatch} applies, completing the proof.
\end{proof}

The next lemma now follows from the previous lemma by an exhaustive colouring argument.

\begin{lemma}\label{l:indptbound_critwalks} Let $G$ be a contraction-minimal graph in $\G_{\rm pl}(\M,\alpha)$. Then there is an independent bound for the number of critical planar type walks in $G$.
\end{lemma}

\begin{proof} Fix a nonempty subset of the nontriangular faces 
of $G$ and let $\U$ be the set of open disc faces $U$ of critical walks that contain the nontriangular faces of the subset and no others. There are finitely many such subsets, so it will be sufficient to show that $|\U|$ has an independent bound.
Choose an initial set $U_0$ in $\U$ and view each of the triangular faces of $G$ that it contains as coloured or labelled, with the remaining triangular faces uncoloured. For any set $U$ in $\U$ exactly one of the following possibilities holds   This is because the open sets in $\U$ are open discs and contain the same nontriangular faces.   : 
\medskip

(i)  $U\subseteq U_0$.
\medskip

(ii) There is an uncoloured triangular face $T$ in $U\backslash U_0$ with an edge contained in the boundary of $U_0$. 
\medskip

If (i) holds then by the previous lemma there is an independent bound for $|\U|$. If (ii) holds then we proceed with the following exhaustive colouring argument to obtain the same conclusion, as desired.

There is an independent bound for the length of the boundary walk of $U_0$ so there is an independent bound for the number of the  incident triangular faces $T$ satisfying (ii). Let $\T_1=\{T_1, \dots , T_{r_1}\}$ be the set of these uncoloured faces and let $U_1, \dots , U_{r_1}$ be any choice of sets in $\U$ with $T_i\subset U_i$ for each $i$. 
Colour every triangular face  in $U_i\backslash U_0$, for $1\leq i \leq r_1$. Note especially the fact that \emph{every} set $W$ in $\U$ with $W\backslash U_0$ nonempty now has at least one fewer uncoloured edges. 

Since $r_1$ has an independent bound there is an independent bound for number of edges in the boundary of the closed union, $K$ say, of the sets $U_i$, for $0\leq i \leq r_1$.
Consider now the triangular faces, denoted  $T_{r_1+1}, \dots , T_{r_2}$, that are incident to the boundary edges of $K$ and are uncoloured and lie in some set in $\U$. If there are no such faces then, by the previous lemma $\U$ has an independent bound, otherwise choose sets $U_j$ that contain $T_j$, for $r_1+1, \dots , r_2$, where $r_2$ has an independent bound. Colour all the triangular faces contained in the new sets $U_j$. 
Following this colouring, every set $U$ in $\U$ is either fully coloured or has fewer uncoloured triangular faces than it had before this colouring.

It follows from Lemma \ref{l:boundforU_silverBullet} that after a finite number of repetitions of this selection and colouring process stops, with an independent bound on the number of repetitions. Thus every set of $\U$ is contained in the union of the resulting sets $U_i$ and the size of this union has an independent bound. Since there are finitely many subsets of nontriangular faces, the lemma follows.
\end{proof}

The proof of the next proposition follows the proof scheme in Section \ref{s:theBEthm} of the Barnette-Edelson Theorem. 

\begin{prop}\label{p:bound_on3cyclesG_pl}
There is an independent bound for the number of essential 3-cycles in a contraction-minimal graph in $\G_{\rm pl}(\M,\alpha)$.
\end{prop}

\begin{proof} 
 Assume that $\M$ has genus $g$ and is not the sphere or the  projective plane and that $G$ is a contraction minimal graph in $\G_{\rm pl}(\M,\alpha)$. By Lemmas \ref{l:boundforU_silverBullet}, \ref{l:indptbound_critwalks}
we may assume that there exists an essential 3-cycle $c$ and that no vertex of $c$ lies on the boundary walk of a nontriangular face. 
For the induction step of the proof assume that for all compact 2-manifolds $\N$ of genus less than  $g$ there is an independent bound for the number of vertices of a contraction-minimal embedded graph in $\G_{\rm pl}(\N,\alpha)$.

 As in the proof of Theorem \ref{t:main}, cut and cap $\M$ by $c$ to create the surface $\M_1$ (or pair of surfaces $\M_1, \M_2$) and the embedded graph $G_1$ (or pair $G_1, G_2$). Suppose first that there is a single connected surface $\M_1$ with embedded graph $G_1$ and that $G$ is orientable. We have $f(G)=\alpha$ and so $f(G_1)=\alpha+6$.
The embedded graph $G_1$ in $\M_1$ is cellular and it satisfies the planar girth inequalities since any critical walk $d$ for $G_1$ with open disc face $U$ containing nontriangular faces of $G_1$ is also a critical walk for $G$ for the same nontriangular faces as they appear in $G$. Thus $G_1$ is in $\G_{\rm pl}(\M_1,\alpha+6)$.
If $e$ is an edge of $G_1$ that is not an edge of the capping 3-cycle and if $e$ is contractible, with $G_1/e$ in $\G_{\rm pl}(\M_1,\alpha+6)$, then, as an edge of the contraction minimal graph $G$ it must lie on an essential 3-cycle or critical walk that passes through a vertex of $c$. Since the degrees of the vertices of $G$ have an independent bound, by Lemma \ref{l:degreeboundGirth}, it follows that the number of such edges has a uniform bound. The theorem follows in this case, as in the proof of Lemma \ref{l:simpleLemma} for almost minimal triangulations.

The same argument holds for the case that $\M$ is orientable and the cut leads to 2 surfaces $\M_1, \M_2$, and the nonorientable cases are entirely similar. 
\end{proof}

We may now complete the proof of Theorem \ref{t:finiteforgirth}, that there are finitely many contraction-minimal embedded graphs $G$, with $f(G)=\alpha$, that satisfy the planar type girth inequalities.
Indeed, there is an independent bound for the number of edges that are boundary edges for a nontriangular face, and every other edge of $G$ lies on an essential 3-cycle or a critical walk and Lemma \ref{l:indptbound_critwalks} and Proposition \ref{p:bound_on3cyclesG_pl} provide independent bounds for these.


\section{Higher genus girth conditions and (3,6)-tight graphs}\label{s:generalGIand}

In this section we define and discuss higher genus girth conditions and characterise cellular $(3,6)$-tight embedded graphs in these terms.

Let $G\subset \M$ be an embedded graph in a compact connected 2-manifold $\M$. A \emph{superface} $U$ of $G$ is a face of a subgraph $K$ of $G$ where $U$ is not dense and $K$ has no vertices of degree 0 or 1. The nondensity condition holds if and only if the complement of $U$ contains a face of $G$. 

Let $G_U$ be the embedded subgraph given by the vertices and edges of $G$ that lie in the closure $U^-$ of the superface $U$. Also, let $G_U^c$ denote the  embedded graph 
determined by the edges of $G$ that are disjoint from $U$. 
In particular, $G=G_U \cup G_U^c$ and $G_U \cap G_U^c$ is the embedded graph $\partial U$ whose edges are the boundary edges of $U$. We assume that $\alpha \geq 3$.



A \emph{balanced} superface of $G$ is a superface $U$ for which the complement of the closure of $U$ is nonempty and connected. A \emph{simple} superface $U$ is a superface whose boundary consists of disjoint cycles. 
For a simple superface $U$ with disjoint cycles $d_1, \dots , d_s$ forming the boundary, the closure of $U$ is a surface with boundary, and so the genus of $U$, denoted $g(U)$, is well-defined as the genus of this closure, $U^-$. Alternatively, $g(U)$ is the genus of the surface $\M_U$ obtained by attaching a disc to  each boundary cycle.

 The \emph{reduced genus} $g_r(\M)$ is defined to be $g(\M)$ if $\M$ is oriented and to be $g(\M)/2$ if $\M$ is not oriented. If $U$ is a balanced simple superface with complementary superface $W$, and boundary cycles $d_1, \dots , d_s$ then we have the following addition formula.
\[
g_r(\M)=g_r(\M_U)+g_r(\M_W)+(s-1).
\]
This may be verified by considering a triangulation of $\M$ incorporating the edges of the boundary cycles of $U$ and the associated equation for the Euler characteristic, namely
\[
\chi=v-e+f=2-2g_r(\M).
\]

It follows from Lemma \ref{l:alphacyclebounds} that if $G$ is a cellularly embedded graph with  $f(G)=3v-e=\alpha $ and nontriangular faces, $U_1, \dots , U_n$, with closed boundary walks $c_1, \dots,  c_n$, then 
\[
\sum_k (|c_k|-3) = 6g_r(\M) + f(G)-6.
\]
We refer to this as the \emph{face walk equation} for $G$. It shows  that there is an independent bound for the lengths $|c_k|$ and for $n$, depending only on the genus of $\M$ and $\alpha$. 


If $U$ is a balanced and simple superface of $G$  with complementary superface $W$ (that is, the complement of the closure of $U$) then we may consider the embedded subgraph $G_W$ also as an embedded graph in the 2-manifold $\M_W$.
Then the face walk equation for $G_W\subset \M_W$ takes the form
\[
\sum_{k=1}^s (|d_k|-3) + \sum_{k\in I(W)} (|c_k|-3) = 6g_r(\M_W) + f(G_W)-6.
\]
where $I(W)$ is the set of indices $k$ for which $U_k$ is contained in $W$. 

\begin{lemma}\label{l:GIforSymmetric_UKW_CYCLES_B}
Let $G$ be a cellular $\M$-embedded graph with $f(G)=\alpha$, let $U$ be a balanced simple superface of $G$ with complementary superface $W$, and let $d_1, \dots , d_{s}$ be the common boundary cycles of $U$ and $W$. Then the following inequalities are equivalent.
\medskip

(i) \quad  $f(G_W) \geq \alpha$.

(ii) 
\[
\sum_{k=1}^s (|d_k|-3) \geq 6g_r(W)- \sum_{k\in I(W)} (|c_k|-3) +(\alpha-6).
\]

(iii) 
\[
\sum_{k=1}^s (|d_k|-3) \geq \sum_{k\in I(U)} (|c_k|-3)-6 (g_r(U)+s-1)-(\alpha-6).
\]
\end{lemma}

\begin{proof}
That (i) is equivalent to (ii) follows from the face walk equation for $G_W$ in $\M_W$. The equivalence of (ii) and (iii) follows from the face walk equation for $G$ together with the addition formula for the reduced genus.
\end{proof}


A balanced simple superface $U$ in $G$ with boundary cycle set $\{d_1, \dots , d_s\}$  is said to satisfy the \emph{higher genus girth inequality} if the inequality in Lemma \ref{l:GIforSymmetric_UKW_CYCLES_B}(iii) holds. 
The lemma ensures that this is equivalent to the $\alpha$-sparsity of the complementary graph $G_W$.
Also $U$ is a \emph{critical superface} of $G$ if this inequality is an equality. In this case we also say that the cycle set is a \emph{critical cycle set} for $U$. We now define these properties for a general superface using an appropriate definition of reduced genus.

A general superface $U'$ of an embedded graph $G\subset \M$ has a set of closed walks, say $\{d_1', \dots , d_s'\}$.
Also $U'$ is the face of a subgraph $H$ with no vertices of degree 0 or 1, and so consecutive edges of the boundary walk, say $xv, vy$, are distinct with $x\neq y$. We may therefore effect a vertex-splitting move at $v$ that creates a new vertex $w$ located in $U'$ together with embedded edges $xw, vw, yw$ and the facial 3-cycles $xvw, vwy$. 
This move creates a new $\M$-embedded graph $G_1$ containing $G$ together with a subgraph $H_1$ and face $U'_1\subset U'$. Moreover $w$ appears once only on the boundary walk for $U_1'$
and appears on no other boundary walk of $H_1$. Repeating
such vertex splitting operations leads to a simple superface $U$, with $U\subset U'$, that is the face of an embedded graph $H_n \subset \M$ whose closed boundary walks are cycles.
Define $U'$ and  $\{d_1', \dots , d_s'\}$ to be \emph{critical}
if $U$ or, equivalently, its set of boundary cycles, say $\{d_1, \dots , d_s\}$,  is critical.

Noting that in general the closure of a superface $U'$ is not a surface with disjoint boundary cycles we \emph{define} the reduced genus $g_r(U')$ to be equal to $g_r(U)$. 

In view of the Lemma \ref{l:GIforSymmetric_UKW_CYCLES_B} and noting that the lengths of $d_i$ and $d_i'$ agree for each $i$, we now give the definition of a general  critical superface and  the associated notion of a \emph{critical edge} of $G$.

\begin{definition}\label{d:criticalsuperface}
(i) A superface $U$ of the (3,6)-tight embedded graph $G$, with boundary walks $\{d_1, \dots , d_s\}$, satisfies the \emph{higher genus girth inequality} if
\[
\sum_{k=1}^s (|d_k|-3) \geq  \sum_{k\in I(U)} (|c_k|-3)-6 (g_r(U)+s-1)-(\alpha-6)
\]
and is a \emph{critical superface} if equality holds.

(ii) An edge $e$ of $G$ is a \emph{critical edge} if it is incident to 2 triangular faces and at least one of the following holds: 
(a) $e$  lies on an essential 3-cycle, that is, one whose curve is not null-homotopic.  
(b) $e$ lies on one of the boundary walks of a critical superface. 
\end{definition}

Suppose now that $U$ is a superface that is not balanced, so that the complementary open set $W=\M\backslash U^-$ has components $W_1, \dots , W_\kappa$, with $\kappa\geq 2$. Then, with the same proof we have the following generalisation of Lemma \ref{l:GIforSymmetric_UKW_CYCLES_B}. 

\begin{lemma}\label{l:GIequivalence_general}
Let $G$ be a cellular $\M$-embedded graph with $f(G)=\alpha$ and let $U$ be a superface of $G$ with boundary walks $d_1, \dots , d_s$ and where the complement $W$ of $U^-$ has connected components $W_1, \dots , W_\kappa$. Then the following statements are equivalent.
\medskip

(i)\quad   $\sum_{i=1}^\kappa f(G_{W_i}) \geq \kappa\alpha$.
\medskip

(ii)  $U$ satisfies the higher genus girth inequality. 
\end{lemma}


\begin{lemma}\label{l:balancedisenough}
Let $G$ be a cellular $\M$-embedded graph with $f(G)=\alpha$ and suppose that $f(G_W) \geq \alpha$ for all balanced superfaces of $G$. Then $f(G_U)\geq \alpha$ for each superface $U$ of $G$.
\end{lemma}

\begin{proof} 
To prove the equivalence we may assume that the boundary of $U$ consists of disjoint cycles.
To see this note that the superface $U'$ for the embedded graph $G'\supset G$ obtained by vertex splitting on the boundary $\partial U$ satisfies 
$f(G_{U'})=f(G_U)$.
 
Let 
$W_1, \dots ,W_\kappa$ be the connected components of the complement of $U^-$ and let
\[
U_i'=U\cup \bigcup_{j\neq i} W_j.
\]
Then $U_i$ is a balanced superface and $f(G_{U_i})\geq \alpha$ by assumption.
Let $a=f(G_U)$ and $a_i=f(G_{W_i})-f(\partial W_i)$, for $1\leq i \leq \kappa$. Then
\[
a+\sum_{i=1}^\kappa a_i =f(G)=\alpha \quad \mbox{and} \quad 
a+\sum_{j\neq i} a_j =f(G_{U_i})\geq \alpha.
\]
Adding the inequalities and substituting $\alpha-a$ for the sum  $\sum_j a_j$ gives $a\geq \alpha$ as desired.
\end{proof}

Let $\gamma$ be the genus of $\M$ and define the families
\[
\G_0(\M,\alpha)\supseteq \G_1(\M,\alpha) \supseteq \dots \supseteq \G_\gamma(\M,\alpha)
\]
where $\G_k(\M,\alpha)$ consists of the $\M$-embedded cellular graphs with $f(G)=\alpha$ that satisfy the girth inequalities for balanced superfaces $U$ with $g(U) \leq k.$

\subsection{Girth conditions and $(3,6)$-sparsity}

We now obtains a girth inequality characterisation of cellularly embedded $(3,6)$-tight graphs.

\begin{definition}\label{d:generalgirthineq} An $\M$-embedded simple graph $G$  with the sparsity count $f(G)=\alpha$ satisfies the \emph{superface sparsity inequalities} if $f(G_U)\geq \alpha$  for all superfaces $U$. 
\end{definition}


The following theorem is due to Shakir \cite{sha} and it gives an insight into the $(3,6)$-tightness of $G$. In fact it is the ``no blocks case" of a more general equivalence for the higher genus analogues of the block and hole graphs described in the introduction.  

\begin{thm}\label{t:generalgirthequals}
Let $G$ be a cellularly embedded graph in the connected compact 2-manifold $\M$. Then the following are equivalent. 

(i) $G$ is (3,6)-tight.

(ii) $G$ satisfies the sparsity count $f(G)=6$ and the associated superface sparsity inequalities, $f(G_U)\geq 6$ for all superfaces $U$ of $G$.
\end{thm}

\begin{proof}
That (i) implies (ii) is immediate. Suppose that (ii) holds and (i) is not true. Then there is a maximal subgraph $K$ of $G$, with at least 3 vertices and  $f(K)\leq 5$. 
Let $U$ be a face of $K$ and suppose first that the boundary walks of $U$ consist of disjoint cycles. Then the complement of the closure of $U$ is the union of components, say $W_1, \dots ,W_r$, that have disjoint closures. In particular, $G_U^c$ is the union of the  disjoint subgraphs,
$G_{W_1}, \dots , G_{W_r}$ and $f(G_U^c)\geq 6r$, by the superface sparsity inequalities. 

In the general case we perform, once again, vertex splitting moves on the vertices of each boundary walk of $W_1, \dots ,W_r$, creating new vertices and edges in each $W_i$, and superfaces $W_i'\subset W_i$ for the embedded graph $K'$ determined by $K$ and the vertex-splitting moves. We can also view these moves as defined on $G_U^c$, determining an embedded graph $(G_U^c)'$. Since 
$f((G_U^c)')$ is sum of $f(G_{W_i'}^c)$, for $1\leq i \leq r$, we have $f((G_U^c)')\geq 6r$. Also $f((G_U^c)')=f(G_{U}^c)$ and so in all cases we have $f(G_U^c)\geq 6$.


Since  $f(G)=f(G_U) + f(G_U^c) - f(G_U \cap G_U^c)$ it follows that
$f(G_U) - f(G_U \cap G_U^c)\leq 0$.
Also, $\partial U = K \cap G_U=G_U \cap G_U^c$ and
$f(K\cup G_U)=f(K) + (f(G_U) - f(K \cap G_U)),$
and so $f(K\cup G_U)\leq 5$. Thus $K=K\cup G_U$ by the maximality of $K$. Since this is true for all faces $U$ of $K$ it follows that $K=G$, a contradiction, as required. Indeed, simple counting shows that $K$ is not a tree and so has at least 2 faces $U_1, U_2$, and $G$ is the union of $G_{U_1}^c$ and $G_{U_2}^c$.
\end{proof}

An $\M$-embedded graph $G$ is \emph{$(3,3)$-tight} if the sparsity count $f(G)=3$ holds and $f(G')\geq 3$ for all subgraphs $G'$. We write $\F(\M,3)$ for the family of cellularly embedded $(3,3)$-tight graphs.
The argument of the previous theorem shows similarly that $G$ is in $\F(\M,3)$ if and only if
$f(G)=3$ and  $f(G_U)\geq 3$ for all superfaces $U$ of $G$.

\begin{thm}\label{t:Gr_equals_F}
Let $\M$ be a connected compact 2-manifold with genus $\gamma$. Then $\G_\gamma(\M,6)=\F(\M,6)$ and $\G_\gamma(\M,3)=\F(\M,3)$. 
\end{thm}

\begin{proof} 
Let $U$ be a superface of $G$ where the complement of the closure of $U$ has components $W_1, \dots , W_\kappa$. If $G$ is in $\F(\M,6)$ then $f(G_{W_i})\geq 6$ for each $i$ and so the inequality of Lemma \ref{l:GIequivalence_general} (i) holds. By the equivalences given in the lemma $U$ satisfies the higher genus girth inequalities and so  $G$ is in $\G_\gamma(\M,6)$.


For the converse let $G$ belong to $\G_\gamma(\M,6)$ and let $W$ be a balanced superface of $G$. If $U$ is the complementary superface then $U$ satisfies the higher rank girth inequality and so Lemma \ref{l:GIequivalence_general} (ii) holds for $U$. By the equivalence of the lemma, with $W_1=W, \kappa=1$, we have $f(G_W)\geq 6$. Now, by Lemma \ref{l:balancedisenough} for $\alpha=6$, $G$ satisfies the  superface sparsity inequalities and so, since $f(G)=6$, it follows from Theorem \ref{t:generalgirthequals} that $G$ is $(3,6)$-tight.
The same argument, with trivial adjustments, shows that $\G_\gamma(\M,3)=\F(\M,3)$. 
\end{proof}




\section{Contraction-minimality and sparse graphs.}\label{s:36minimalsThm}

We now obtain analogues of the Barnette-Edelson finiteness theorem for partial triangulations of a 2-manifold satisfying higher genus girth conditions. This leads to the finiteness of contraction-minimal graphs in sparse graph families, such as $(3,6)$-tight graphs and $(3,3)$-tight graphs that feature in  the rigidity theory of bar-joint frameworks in $\bR^3$.
We assume that $\alpha\geq 3$.

\begin{lemma}\label{l:degreeBoundHighGenus}
Let $G$ be a contraction-minimal embedded graph in $\G_\gamma(\M,\alpha)$. Then there is an independent bound for the degree of a vertex.
\end{lemma}

\begin{proof}
The proof of Lemma \ref{l:degreeboundGirth} holds without change with the understanding that critical walks $c_e$ are general genus critical walks.
\end{proof}

\begin{lemma}\label{l:indptBoundForGenus0}
Let $G$ be a contraction-minimal embedded graph in $\G_\gamma(\M,\alpha)$ and let $U_0$ be a critical superface of genus $0$. Then there is an independent bound for $|V\cap U_0|$.
\end{lemma}

\begin{proof}
The argument in the proof of Lemma \ref{l:boundforU_silverBullet} applies directly.
\end{proof}

\begin{lemma}\label{l:indptBoundForGenusGeneral}
Let $G$ be a contraction-minimal embedded graph in $\G_\gamma(\M,\alpha)$ and let $U$ be a critical superface of general genus. Then there is an independent bound for $|V\cap U|$.
\end{lemma}

\begin{proof} The proof is by induction on the genus. Noting the previous lemma for genus $0$, let $G$ be contraction-minimal and assume that an independent bound exists for the number of edges in the closure of a superface with genus less than $g$.  We show that there is a such a bound for a superface of genus $g$. We may assume that one of the following is true.
\medskip

(i) There is a boundary walk of a critical superface that is essential, that is, that is not null-homotopic in $\M$.

(ii) There is an essential 3-cycle. 

\medskip
To see this suppose that (ii) does not hold.  Then, by contraction-minimality, every nonboundary edge of $G$ is on a boundary walk for a critical superface. If every such boundary walk is null-homotopic then a general genus critical superface has all its boundary walks  null-homotopic. Thus the components of the complement of $U^-$ are  genus 0 superfaces, and there is an independent bound for these, by Lemma \ref{l:indptBoundForGenus0}. Thus there is an independent bound for the number of critical edges in this case. So we can assume that (i) or (ii) holds and we let $c_{\rm ess}$ be the essential walk or 3-cycle.

We now consider the division of $\M$ by $c_{\rm ess}$ that is analogous to the division in the induction scheme in Section \ref{s:theBEthm}.

Consider first the case that $\M$ is orientable. Assume moreover that the complement $\M\backslash c_{\rm ess}$ is not connected and so has 2 components, namely $U_0$, the superface for $c_{\rm ess}$, and $W$ say. Since $c_{\rm ess}$ is essential we have $g(U_0)< g$ and $g(W)<g$. 
Let $e$ be an edge in a critical superface $U$ of genus $g$. Then $e$ lies on a closed walk $c$ that is either an essential 3-cycle or a boundary walk of a critical superface $U_c$.
If $c$ happens to lie in $G_U$ or $G_W$ then, by the induction hypothesis, $U_c$ has an independent bound for the edges in its closure. The only other possibility requires that $c$ contains a vertex of $c_{\rm ess}$. By Lemma \ref{l:degreeBoundHighGenus}, and the bounded lengths of critical walks, it follows that there is an independent bound for the number of such critical walks. Combining these facts it follows that there is an independent bound for the number of edges $e$. This completes the induction step and so the proof is complete in this case. The case that $\M\backslash c_{\rm ess}$ is connected is similar, as are the nonorientable cases.
\end{proof}

\begin{lemma}\label{l:indptBoundForCriticalWalks}
Let $G$ be a contraction-minimal embedded graph in $\G_\gamma(\M,\alpha)$. Then there is an independent bound for the set of edges that lie on the boundary of a critical superface.
\end{lemma}

\begin{proof} 
This follows from Lemma \ref{l:indptBoundForGenusGeneral} and the exhaustive colouring argument in the proof of Lemma \ref{l:indptbound_critwalks}. 
\end{proof}

\begin{thm}\label{t:mainGgamma}
Let $\M$ be a compact connected 2-manifold. Then there are  finitely many contraction-minimal graphs in $\G_\gamma(\M,\alpha)$.  
\end{thm}

\begin{proof}
This follows from Proposition \ref{p:bound_on3cyclesG_pl} and Lemma \ref{l:indptBoundForCriticalWalks} since every edge of a contraction minimal graph in $\G_\gamma(\M,\alpha)$ is either a critical edge or the edge of the boundary of a nontriangular face.
\end{proof}

\subsection{Contraction minimal graphs and 3-rigidity} 

The following theorem is an immediate corollary of Theorems \ref{t:mainGgamma}, \ref{t:Gr_equals_F}. 

\begin{thm}\label{t:finiteFor36_33}
Let $\M$ be a compact 2-manifold. Then there are finitely many contraction-minimal graphs in $\F(\M,6)$ and $\F(\M,3)$.
\end{thm}

Recall that a \emph{bar-joint framework} in $\bR^3$ is a pair $(G,p)$ where $G=(V,E)$ is a simple graph and $p:V\to \bR^3$ is a spatial realisation of its vertices, giving a spatial realisation of $G$. See \cite{gra-ser-ser} for  example. Also, $(G,p)$ is a \emph{generic framework} if the set of coordinates of $p(v)$, for $v$ in $V$, is an algebraically independent set, and $G$ is a \emph{3-rigid} graph if its generic framework admits no nontrivial  motions preserving the lengths,  $\|p(v)-p(w)\|_2$, of the spatial edges (framework bars) for all edges $vw$ in $E$. Furthermore a graph is \emph{minimally 3-rigid} if any proper spanning subgraph of $G$ fails to be 3-rigid. A necessary condition for minimal 3-rigidity is that $G$ is $(3,6)$-tight. It remains a significant open problem to determine necessary and sufficient combinatorial conditions for minimal 3-rigidity.

An important fact in framework rigidity theory, mentioned in the introduction, is that vertex-splitting preserves minimal 3-rigidity. Accordingly, if it is known that a particular family of graphs $G$ derives from a contraction-minimal graph $H$ by repeated vertex splitting then it remains only to confirm the minimal 3-rigidity of $H$ to be assured of the minimal 3-rigidity of the graphs $G$.

\begin{example} {\bf Projective plane graphs.} For the projective plane it was shown by direct methods, in Kastis and Power \cite{kas-pow}, that there are 8  contraction-minimal (3,6)-tight cellular graphs and these are shown below, where diametrically opposite vertices and edges are identified. These base graphs are 3-rigid and so all $(3,6)$-tight graphs that admit a cellular embedding in the projective plane are minimally 3-rigid.

\begin{center}
\begin{figure}[ht]
\centering
\includegraphics[width=2cm]{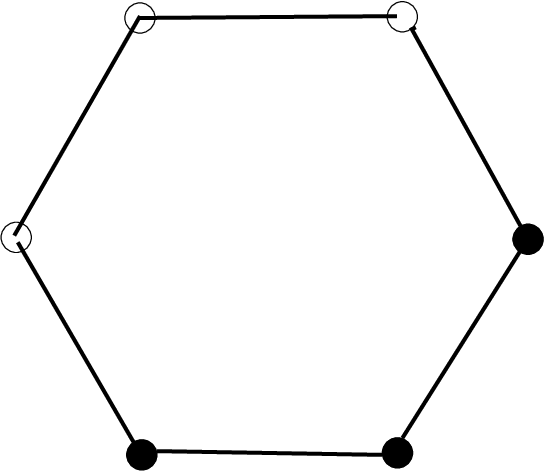}\quad 
\includegraphics[width=2cm]{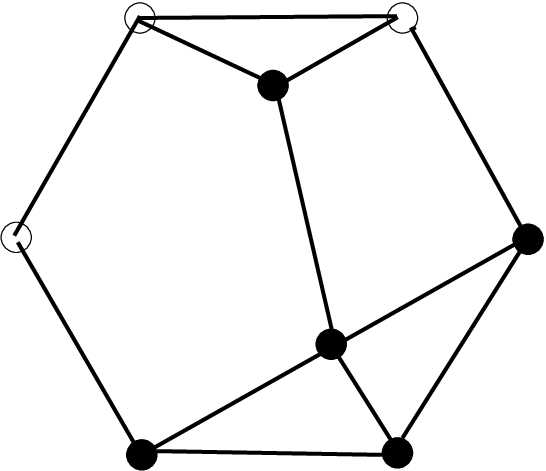}\quad 
\includegraphics[width=2cm]{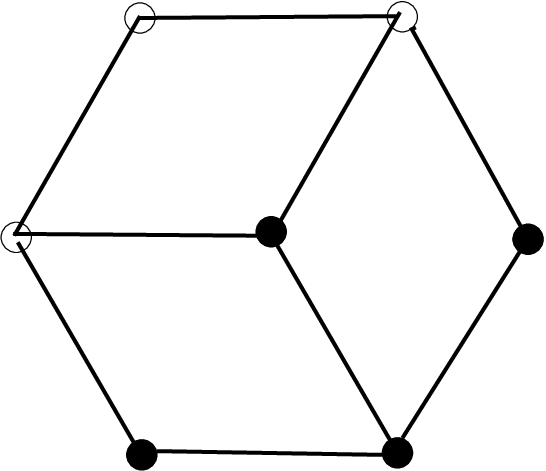}\quad 
\includegraphics[width=2cm]{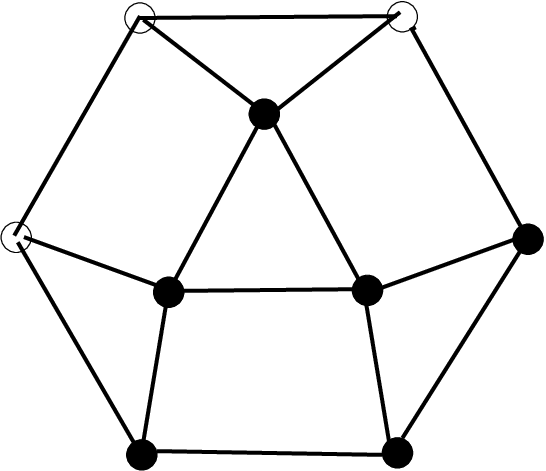}\quad
\end{figure}
\end{center}
\begin{center}
\begin{figure}[ht]
\centering

\includegraphics[width=2cm]{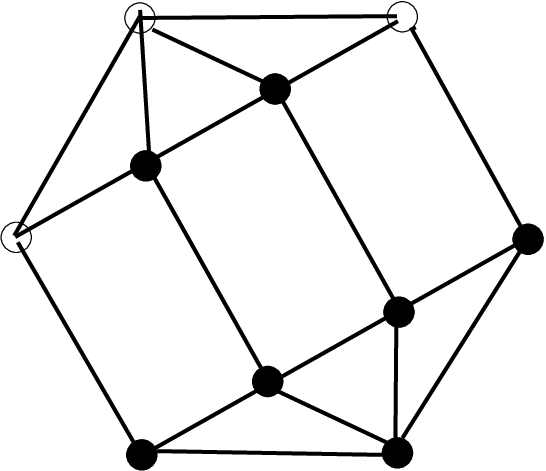}\quad
\includegraphics[width=2cm]{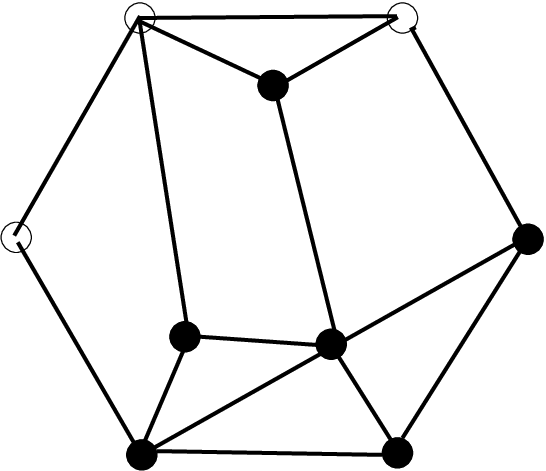}\quad
\includegraphics[width=2cm]{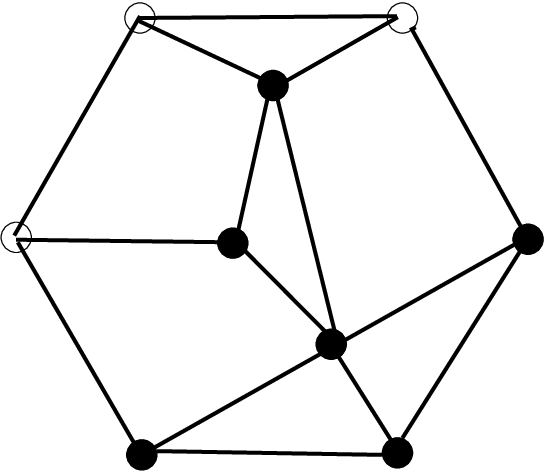}\quad 
\includegraphics[width=2cm]{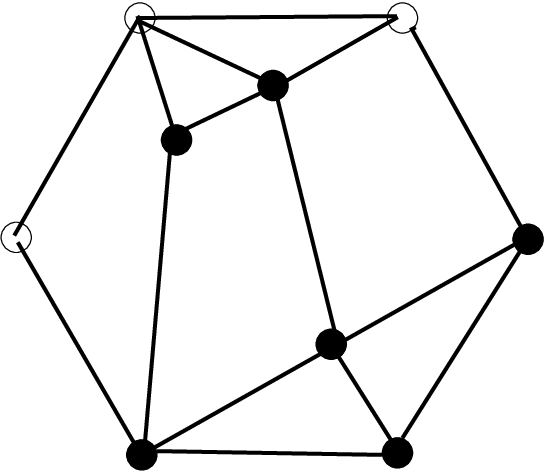}\quad 
\caption{The 8 contraction-minimal cellular (3,6)-tight graphs in the projective plane.} 
\label{f:essentialIrreducibles}
\end{figure}
\end{center}

\end{example}

\begin{example} {\bf Torus graphs.}
For $(3,6)$-tight torus graphs it follows from Lemma \ref{l:alphacyclebounds} that there are 11 possibilities for the set of ordered lists of nontriangular face sizes, namely $\{4,4,4,4,4,4\}$, $\{5,4,4,4,4\}$, $\{5,5,4,4\}$, $\{5,5,5\}$, $\{6,4,4,4\}$, $\{6,5,4\}$, $\{6,6\}$, $\{7,4,4\}$, $\{7,5\}$,
$\{8,4\}$ and $\{9\}$. Also there are $(3,6)$-tight torus graphs with two or more nontriangular faces that are not 3-rigid. On the other hand, in Cruickshank, Kitson and Power \cite{cru-kit-pow-2} it was shown by direct methods that there are 2 contraction-minimal $(3,6)$-tight cellularly embedded graphs for the torus that have a single nontriangular face. These are shown in Figure \ref{f:singleholepair} and are minimally 3-rigid. Accordingly,
all $(3,6)$-tight graphs that admit a cellular embedding in the torus with a single nontriangular face are minimally 3-rigid. 

 \begin{center}
\begin{figure}[ht]
\centering
\includegraphics[width=2cm]{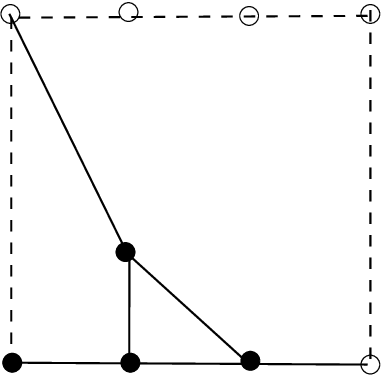}\quad  \quad \quad 
\includegraphics[width=2cm]{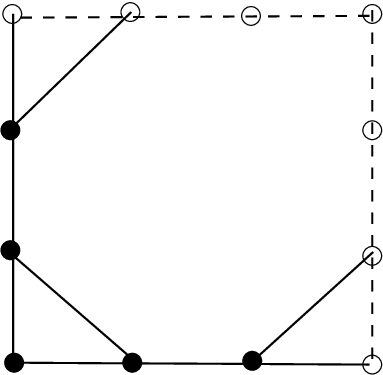}
\caption{The minimal (3,6)-tight cellular graphs for the torus with a single nontriangular face.}
\label{f:singleholepair}
\end{figure}
\end{center}

 Simple exploration shows that there are more than 35 minimals with only 4-sided nontriangular faces, and this suggests that there are more than 100 contraction-minimal graphs. It would be of interest to determine these and their 3-rigidity.
Figure \ref{f:contrastingMinimals} shows 4 contrasting examples with 6 quadrilateral faces.

 \begin{center}
\begin{figure}[ht]
\centering
\includegraphics[width=1.6cm]{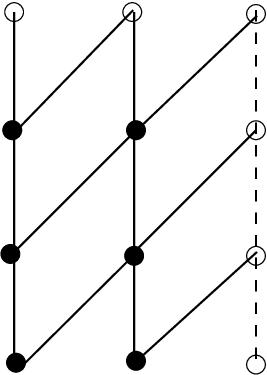}\quad \quad 
\includegraphics[width=1.75cm]{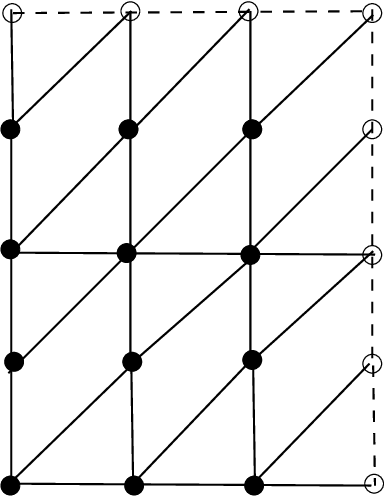} \quad \quad
\includegraphics[width=2.2cm]{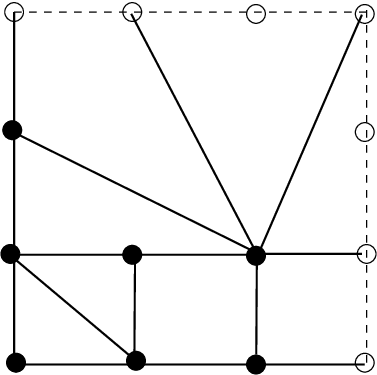}\quad \quad
\includegraphics[width=3cm]{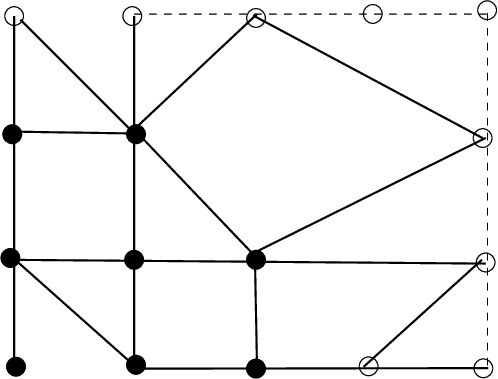}
\caption{Some minimal graphs for the torus with 6 four-sided faces.}
\label{f:contrastingMinimals}
\end{figure}
\end{center}

\end{example}

\subsection{NonEuclidean minimal rigidity.}
One can also consider the minimal 3-rigidity of bar-joint frameworks in $\bR^3$ with respect to non-Euclidean norms, such as the norms $\|\cdot \|_q$, for $1<q\neq 2<\infty$. See Kitson and Power \cite{kit-pow} and Dewar, Kitson and Nixon  \cite{dew-kit-nix} for example. The necessary condition for such minimal $3$-rigidity of a framework $(G,p)$ is that the simple graph $G$ is $(3,3)$-tight. It remains an open problem from \cite{kit-pow} whether this condition is also sufficient in general.

For the sphere, in view of Lemma \ref{l:alphacyclebounds} there are no cellularly embedded $(3,3)$-tight graphs. For the projective plane they are precisely the graphs of triangulations, and it is shown in \cite{dew-kit-nix} that these are the minimally 3-rigid projective plane graphs for $\|\cdot\|_q$. For the genus 1 torus $\M$, to show the sufficiency of $(3,3)$-tightness for minimal 3-rigidity for $\|\cdot\|_q$ it would be sufficient, by Theorem \ref{t:finiteFor36_33}, to identify the contraction-minimal graphs in $\G(\M,3)$ and to determine their $3$-rigidity.

Figures \ref{f:33tightMinimalsA} and \ref{f:33tightMinimalsB} show 5 examples of contraction minimal (3,3)-tight cellular torus graphs with 7 vertices and a single nontriangular face, denoted
$7v1,7v2, \dots ,7v5$.
 \begin{center}
\begin{figure}[ht]
\centering
\includegraphics[width=2.2cm]{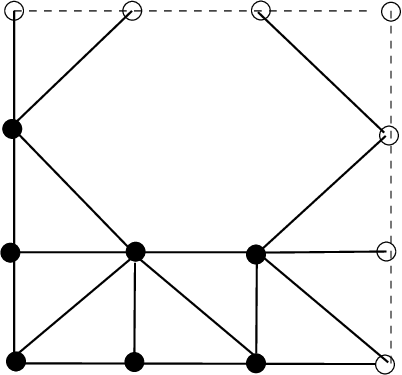}\quad \quad 
\includegraphics[width=2.2cm]{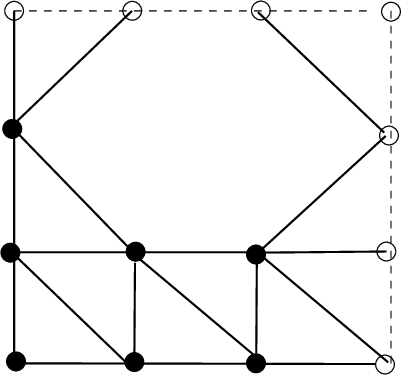}\quad \quad 
\includegraphics[width=2.2cm]{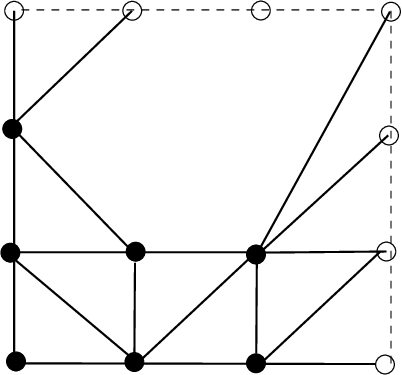}\quad \quad 
\includegraphics[width=2.2cm]{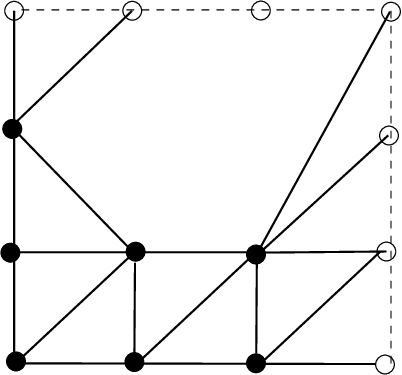}\quad \quad 
\caption{$7v1, 7v2, 7v3, 7v4$, contraction-minimals.}
\label{f:33tightMinimalsA}
\end{figure}
\end{center}

 \begin{center}
\begin{figure}[ht]
\centering
\includegraphics[width=2.2cm]{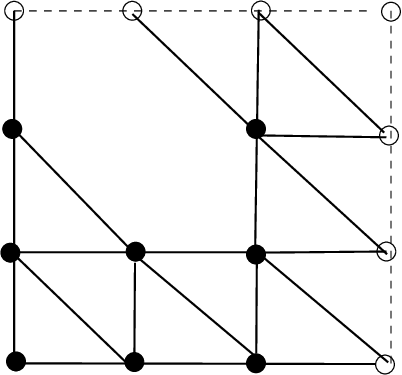}\quad \quad 
\includegraphics[width=2.2cm]{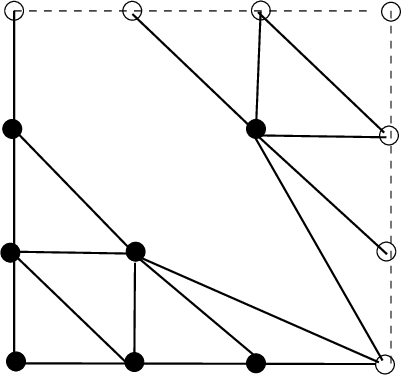}\quad \quad 
\caption{(i) $8v1$ and (ii) its reduction $7v5$, a contraction-minimal.}
\label{f:33tightMinimalsB}
\end{figure}
\end{center}

Consider the \emph{vertex degree signature} of a graph to be the symbol $d_1^{r_1}\dots d_k^{r_k}$ where there are $r_j$ vertices of degree $d_j$. For $7v1, \dots ,7v5$ these symbols  are
\[
6^35^24^2, \quad 65^6,\quad 6^354^3,\quad  65^54,\quad  65^6.
\]
Also, as embedded graphs, $7v2$ and $7v5$ differ in that the pair of vertices that are not boundary vertices are adjacent vertices for $7v2$ and are not adjacent for $7v5$. It follows that these 5 embedded graphs are not homeomorphic. It would be of interest to determine a complete list of contraction-minimal $(3,3)$-tight torus graphs, including those  with a single nontriangular face, and to confirm their minimal 3-rigidity for the nonEuclidean norms $\|\cdot\|_q$.




\bibliographystyle{abbrv}
\def\lfhook#1{\setbox0=\hbox{#1}{\ooalign{\hidewidth
  \lower1.5ex\hbox{'}\hidewidth\crcr\unhbox0}}}

\end{document}